\documentclass[reqno,11pt,twoside]{amsart}

\usepackage{amssymb}

\numberwithin{equation}{section}

\usepackage{xypic} 

\usepackage{mathrsfs}
\usepackage{hyperref}
\hypersetup{ 
colorlinks,
citecolor=black,
filecolor=blue,
linkcolor=black,
urlcolor=blue 
} 
\usepackage{comment}
\usepackage{color}
\xyoption{curve}

\numberwithin{equation}{subsection}
\newtheorem{lemma}[equation]{Lemma}
\newtheorem{lem/def}[equation]{Lemma/Definition}
\newtheorem{proposition}[equation]{Proposition}
\newtheorem{theorem}[equation]{Theorem}
\newtheorem{corollary}[equation]{Corollary}

\theoremstyle{definition}

\newtheorem{definition}[equation]{Definition}
\newtheorem{remark}[equation]{Remark}

\newtheorem{notation}[equation]{Notation}

\newcommand{\beq}{\begin{equation*}}
\newcommand{\eeq}{\end{equation*}}
\newcommand{\beqlbl}{\begin{equation}}
\newcommand{\eeqlbl}{\end{equation}}
\newcommand{\ba}{\begin{array}}
\newcommand{\ea}{\end{array}}
\newcommand{\Hom}{\mathrm{Hom}}

\newcommand{\Ext}{\mathrm{Ext}}

\newcommand{\ot}{\otimes}
\newcommand{\ox}{\otimes}

\newcommand{\xym}{\xymatrix}

\DeclareMathOperator{\coh}{H}
\DeclareMathOperator{\HH}{HH}

\title[Lie bracket on Hochschild cohomology]{An Alternate Approach to the Lie Bracket
on Hochschild Cohomology}
\date{13 March 2015}

\author{Cris Negron}
\address{Department of Mathematics\\University of Washington\\
Seattle, WA 98195, USA}
\email{negron@uw.edu}

\author{Sarah Witherspoon}
\address{Department of Mathematics\\Texas A\&M University\\College Station, TX 77843,
USA}
\email{sjw@math.tamu.edu}
\thanks{The first author was supported by the
NSF Graduate Research Fellowship under grant DGE-1256082.
The second author was partially supported by NSF grant DMS-1101399.}

\begin{document}

\maketitle
\begin{abstract}
We define Gerstenhaber's graded Lie bracket directly on complexes other than
the bar complex, under some conditions, resulting in a practical technique for
explicit computations. 
The Koszul complex of a Koszul algebra in particular 
satisfies our conditions. 
As examples we recover the Schouten-Nijenhuis bracket for a polynomial ring 
and the Gerstenhaber bracket for a group algebra of a cyclic group of prime order.
\end{abstract}

\section{Introduction}

Hochschild cohomology incorporates 
useful information about an algebra, the study of which was begun by
Hochschild~\cite{Hochschild45} and Gerstenhaber~\cite{Gerstenhaber63}.
In low degrees one finds the center of the algebra, 
derivations, infinitesimal deformations,
and obstructions in algebraic deformation theory. 
Vanishing in high degrees is equivalent to smoothness in commutative 
settings~\cite{AvramovIyengar05,AvramovVigue-Poirrier92}.
Noncommutative algebras can behave quite 
differently~\cite{BuchweitzGreenMadsenSolberg}, yet 
analogous notions have also been explored in  noncommutative settings~\cite{Krahmer07}.
Hochschild cohomology is used in support variety theory, 
a tool for studying representations of some types of finite
dimensional algebras~\cite{SnashallSolberg04}.

In spite of its many uses, some of the structure of Hochschild cohomology
remains elusive.
It is a Gerstenhaber algebra, that is, it has both a cup product and a graded
Lie bracket, and the bracket induces graded derivations with respect to the product.
Products are defined  in any number of equivalent ways:
as Yoneda composition of $n$-extensions of bimodules, as composition of maps in
arbitrary projective bimodule resolutions, or as application of a diagonal map
to the tensor product over the algebra of two copies of an arbitrary
bimodule resolution. 
This freedom of choice 
makes the product quite tractable for many algebras. 
Brackets have not been so amenable to study on resolutions
other than the bar resolution where they were historically defined, and thus they are 
more difficult to compute and to use.
Typically one computes cohomology with a resolution other than the bar resolution,
and then translates the bracket from the bar resolution using explicit comparison
maps. These maps are nearly always very cumbersome, and beg for a better approach.

The question about realizing the Gerstenhaber bracket on other resolutions was raised
by Gerstenhaber and Schack~\cite{GerstenhaberSchack88}.
An elegant such realization was given by Schwede~\cite{Schwede98}, based on
Retakh's description of categories of extensions~\cite{Retakh86}.
Hermann~\cite{Hermann14} generalized Schwede's construction of brackets as loops
in an extension category to other suitable exact monoidal categories.
Yet it seems difficult to translate these beautiful constructions into practical
techniques for explicit computations of brackets as may be required, for example,
to answer some questions in algebraic deformation theory. 
Our paper takes a different route to  computational techniques. 

We begin with the observation that there is more than one way to define the graded
Lie bracket on the bar resolution $B$ of an algebra:
We show in Section~\ref{altC} that 
a particular class of chain maps, of graded degree 1, from the tensor
product of two copies of $B$ to $B$,
gives rise to many brackets at the chain level. 
These all induce the Gerstenhaber bracket on cohomology. 
We mimic this construction in Section~\ref{other}  
for other resolutions satisfying some hypotheses. 
We define brackets and prove that these brackets also induce 
Gerstenhaber brackets on cohomology.
One useful condition in particular is when 
the resolution embeds into the bar resolution in such a way
that the diagonal maps commute with the embedding, and
the strongest results follow from this condition (Subsection~\ref{best}).  
Koszul resolutions of Koszul algebras in particular 
satisfy this hypothesis. 

We illustrate our techniques  by recovering
the Schouten-Nijenhuis bracket on polynomial rings in Section~\ref{snbracket}. 
We also give some results under weaker conditions (Subsection~\ref{weaker})
that still may be useful but have the disadvantage of requiring 
a more detailed comparison with the bar resolution.
In Section~\ref{cyclic} 
we show that these techniques may 
be used to recover Gerstenhaber brackets for a group algebra
of a cyclic group of prime order $p$ over a field of characteristic $p$.  (Expressions for such brackets were originally given in the work of Sanchez-Flores~\cite{Sanchez-Flores}.)  
These two well known classes of examples, in Sections~\ref{snbracket} and \ref{cyclic},
serve merely to illustrate our techniques here.
A new class of examples is given in~\cite{GrimleyNguyenWitherspoon}: 
Brackets are computed there for the quantum complete intesections
$\Lambda_q := k\langle x,y\rangle/(x^2, y^2, xy+qyx)$ for various
(nonzero) values of a parameter $q$ in a field $k$.
The algebra structure of Hochschild cohomology of $\Lambda_q$ had been computed
by Buchweitz, Green, Madsen, and Solberg~\cite{BuchweitzGreenMadsenSolberg}.
Grimley, Nguyen, and the second author~\cite{GrimleyNguyenWitherspoon} used the
techniques of the current paper to compute Gerstenhaber brackets directly on 
the Koszul resolution of $\Lambda_q$.
They did not need to know explicit formulas for chain maps between the bar and
Koszul resolutions, as these were not used; it suffices to know existence of such maps
 satisfying some conditions. 
Also in~\cite{GrimleyNguyenWitherspoon} 
is a general result about the Gerstenhaber algebra structure of the
Hochschild cohomology of a twisted tensor product of algebras.
Its proof uses 
techniques from the current paper, showing that these techniques can be useful as
well for algebras that are not Koszul. 


\section{Alternate brackets on the Hochschild complex}\label{altC}

Let $k$ be a field of arbitrary characteristic and let $A$ be a $k$-algebra. 
Let us recall the definitions of the bar
resolution $B$ of $A$ and the Hochschild cochain complex.
We write $\ot$ to mean $\ot_k$. 
\par
Let $TA=T(A)$ denote the graded tensor coalgebra, that is, 
$TA = \oplus_{r\geq 0} (TA)_r$ where $(TA)_r = A^{\ot r}$ and the
coproduct $\Delta : TA \rightarrow TA \ot TA$ is the $k$-linear map
defined by
$$
  \Delta(a_1\ot \cdots\ot a_r) = \sum_{i=0} ^r (a_1\ot\cdots \ot a_i)\ot
    (a_{i+1}\ot \cdots\ot a_r)
$$
for $a_1,\ldots, a_r\in A$. 
As a graded $A$-bimodule, we have $B=A\ox TA\ox A$, with $B_r=A^{\ot (r+2)}$
for each $r\geq 0$. 
We may use the
notation $a\ox x\ox a'$ to denote monomials in $B$, where $a,a'\in A$,
and $x\in TA$.  The differential on $B$ is
\beqlbl
a_0\ox \dots \ox a_{r+1}\mapsto \sum_{0\leq i\leq r}(-1)^ia_0\ox\dots \ox a_ia_{i+1}\ox\dots\ox a_{r+1}.
\label{dffrntl}
\eeqlbl
Note that the comultiplication on $TA$ induces a quasi-isomorphism
\beqlbl
\ba{c}
\Delta:B\to B\ox_A B \\
a_0\ox \dots \ox a_{r+1}\mapsto 
\displaystyle{\sum_i (a_0\ox \dots \ox a_i\ox 1)\ox 
(1\ox a_{i+1}\ox \dots \ox a_{r+1}) } .\ea
\label{diagmap}
\eeqlbl
The map $\Delta$ is coassociative by construction, that is,
$ ( \Delta\ot id) \Delta =  (id \ot \Delta) \Delta$ as chain
maps from $B$ to $B\ot_A B\ot_A B$. 
On monomials, we can write this map symbolically as 
\beq
a\ox x\ox a'\mapsto \sum(a\ox x_1\ox 1)\ox(1\ox x_2\ox a')
\eeq
where the sum runs over all possible ways to factor the monomial $x$.
Let $$C(A) :=\Hom_{A^e}(B,A),$$ where $A^e = A\ot A^{op}$. 
The cup product may be defined at the cochain level via 
the diagonal map $\Delta$:
If $f\in \Hom_{A^e}(B_r,A)$, $g\in \Hom_{A^e}(B_s,A)$, then
\[
   (f\smile g) (a_0\ot\dots\ot a_{r+s+1}) =
    f(a_0\ot \dots\ot a_r \ot 1) g(1\ot a_{r+1}\ot\dots\ot a_{r+s+1}).
\label{cupprod}
\]
We use the notation $|f| = r$ for the degree of $f$ in this case. 
The Gerstenhaber bracket is defined as follows, where we replace
tensor products by commas in function notation for convenience.

\begin{definition}[Standard Gerstenhaber Bracket \cite{Gerstenhaber63}]\label{defn-bracket} 
Let $\circ$ denote the operation $C(A)\ox C(A)\to C(A)$ given on homogeneous elements
$f$ and $g$ by 
\beq
f\circ g ( a_1\ox\dots \ox a_n) = \sum^{|f|}_{j=1} (-1)^{(|g|-1)(j-1)}f(a_1,\dots, 
a_{j-1}, g(a_j,\dots, a_{j+|g|-1}),\dots,a_n), 
\eeq
and define the bracket $[ \ , \ ]$ by
\beq
[f,g]= f\circ g-(-1)^{(|f|-1)(|g|-1)}g\circ f.
\eeq
\end{definition}

The cup product and bracket induce operations on Hochschild cohomology that
enjoy many useful properties, for example,
\begin{equation}\label{derivation}
   [ \bar{f}\smile\bar{g} , \bar{h} ] = 
    [ \bar{f}, \bar{h} ] \smile \bar{g} \ + \ (-1)^{ |\bar{f}| ( |\bar{h}|-1)}
      \bar{f} \smile [\bar{g}, \bar{h} ] ,
\end{equation}
where $f,g,h$ are homogeneous cocycles and $\bar{f}, \bar{g},\bar{h}$ are
their images in Hochschild cohomology. 
See \cite{Gerstenhaber63} for this and other properties.

\begin{lemma}\label{isomorphism} 
As a graded $A$-bimodule, $B\ox_A B\cong A\ox TA\ox A\ox TA\ox A$.
Under this identification, the differential  is given by
\beq
\ba{l}
(a_0\ox\dots\ox a_{j-1})\ox(a_j)\ox(a_{j+1}\ox\dots\ox a_{n+1})\\
\hspace{5mm}\mapsto 
\ba{l}
\sum_{i<j-1} (-1)^i(a_0\ox\dots\ox a_ia_{i+1}\dots)\ox (a_j)\ox(\dots \ox a_{n+1})\\
+(-1)^{j-1}(a_0\ox\dots)\ox (a_{j-1}a_j)\ox(\dots\ox a_{n+1})\\
+(-1)^{j-1}(a_0\ox\dots)\ox (a_ja_{j+1})\ox(\dots\ox a_{n+1})\\
+\sum_{k>j}(-1)^{k-1}(a_0\ox\dots)\ox (a_j)\ox(\dots\ox a_ka_{k+1}\dots\ox a_{n+1}).
\ea
\ea
\eeq
\label{1}
\end{lemma}
\begin{proof}
The first portion of the statement is clear.  The second is an easy
check from the fact that the differential on the tensor complex
$X\ox_A Y$, of any two $A$-bimodule complexes $X$ and $Y$, 
is given by $d(x\ox y)=d(x)\ox y+(-1)^{|x|}x\ox d(y)$.
\end{proof}

We will deconstruct the bracket operation, realize it as a composition
of several maps, and make some changes in the apparent choices involved. 
We will observe that these choices 
do not matter at the level of cohomology,
giving us some freedom in the definition. 
It is this freedom that will allow us, in the next section, 
to define the bracket independently on other cochain complexes
satisfying certain conditions. 

We first define a chain map $F_B: B\ot_A B\rightarrow B$.  
By the isomorphism of Lemma~\ref{isomorphism}, 
elements in the tensor product $B\ox_A B$
may be identified with sums of elements of the form
\beq
a\ox x\ox a'\ox y\ox a'',
\eeq
with $x\in A^{\ot i}$, $y\in A^{\ot j}$ and $a,a',a''\in A$.  
We define 
$F_B: B\ox_A B\to B$ on such monomials as follows:
If $i>0$ and $j>0$, then 
\begin{equation}\label{ogFB}
\begin{array}{rcl}
F_B ( a\ox x\ox a'\ox y\ox a'') &=& 0 , \\ 
F_B ( a\ox a'\ox y\ox a'') & = & aa'\ox y\ox a'' , \\
F_B ( a\ox x\ox a'\ox a'') & = & -a\ox x\ox a'a'' . 
\end{array} 
\end{equation}
In degree 0, 
\beq
F_B(a\ox a'\ox a'')= aa'\ox a''-a\ox a'a'' . 
\eeq
As one can see from the definition, $F_B$ is $0$ on
most of the tensor complex $B\ox_AB$, and is simply given by the actions of $A$
on $B$ for the extremal terms
$B^0\ox_A B^j$ and $B^i\ox_A B^0$.
One may check directly that $F_B$ is a  chain map. 
Alternatively, this follows from Proposition~\ref{dG}
or the general construction given in Section~\ref{best}. 

In the remainder of this article, we use the isomorphism of
Lemma~\ref{isomorphism}, without comment, to identify $B\ot _A B$
with $A\ot TA\ot A\ot TA\ot A$. 

\begin{notation}\label{notation} 
\begin{enumerate}
\item Let $\Delta^{(2)}$ denote the map
\beq
(\Delta\ox id)\Delta=(id\ox \Delta)\Delta:B\to B\ox_A B\ox_A B.
\eeq
\item Let $G:B\ox_A B\to B$ denote the map given
  on monomials by
\beq
G\big((a_0\ox\dots\ox a_{j-1})\ox (a_j)\ox(a_{j+1}\ox\dots \ox a_{n+1})\big)
= (-1)^{j-1}a_0\ox\dots\ox a_{j-1}\ox a_j\ox a_{j+1}\ox\dots \ox a_{n+1} . 
\eeq
\end{enumerate}
\label{notations1}
\end{notation}

Notice that the circle operation $f\circ g$ of Definition~\ref{defn-bracket}, on
Hochschild cochains $f$ and $g$ is precisely the composition
\beq
B \stackrel{\Delta^{(2)}}{\relbar\joinrel\longrightarrow} 
B\ox_A B\ox_A B\stackrel{id_B\ox_A g\ox_A id_B}
{\relbar\joinrel\relbar\joinrel\relbar\joinrel\relbar\joinrel\relbar\joinrel\relbar\joinrel\relbar\joinrel\longrightarrow}
 B\ox_A B\stackrel{G}{\longrightarrow} B\stackrel{f}{\longrightarrow} A .
\eeq
To be clear, in the definition of the map $id_B\ox_A g\ox_A id_B$ one includes ``Koszul signs" so that on elements the map is given by
\beq
(a\ox x\ox a'\ox y\ox a''\ox z\ox a''')\mapsto (-1)^{|x||g|}(a\ox x)\ox g(a'\ox y\ox a'')\ox (z\ox a''').
\eeq
This observation inspires our alternate definition of brackets below
(Definition~\ref{phi-bracket}). 
First we record a crucial property of the map $G$.

\begin{proposition}\label{dG}
Let $d$ be the differential on the  complex $\Hom_{A^e}(B\ox_A
B,B)$.  The map $G\in \Hom_{A^e}(B\ox_A B, B)$ is a contracting
homotopy for $F_B$, that is, $d(G):=d_BG+ Gd_{B\ox_A B}=F_B$.
\end{proposition}
\begin{proof}
Take a monomial 
\beq
(a_0\ox\dots\ox a_{j-1})\ox(a_j)\ox(a_{j+1}\ox\dots\ox a_{n+1})\in B_{j-1}\ox_A B_{n-j}
\eeq
with $j-1, n-j>0$. Applying the formulas given in Lemma \ref{1} and 
Notation~\ref{notation}(2), the function $G d_{B\ox_A
  B}$ sends $(a_0\ox\dots\ox
a_{j-1})\ox(a_j)\ox(a_{j+1}\ox\dots\ox a_{n+1})$ to the element
\beqlbl
\ba{l}
\sum_{i<j-1} (-1)^{(i+j-2)}(a_0\ox\dots \ox a_ia_{i+1}\ox \dots \ox a_{n+1})\\
+(-1)^{(j-1+j-2)}(a_0\ox\dots\ox a_{j-1}a_j \ox\dots\ox a_{n+1})\\
+(-1)^{(j-1+j-1)}(a_0\ox\dots\ox a_ja_{j+1}\ox\dots\ox a_{n+1})\\
+\sum_{j<k}(-1)^{(k-1+j-1)}(a_0\ox\dots\ox a_ka_{k+1}\ox\dots\ox a_{n+1})
\ea
\label{form1}
\eeqlbl
in $B$.  Now $d_B G$ will send that same element in $B\ox_A B$ to
\beqlbl
\ba{l}
\sum_{i<j-1} (-1)^{(j-1+i)}(a_0\ox\dots\ox a_ia_{i+1}\ox\dots \ox a_{n+1})\\
+(-1)^{(j-1+j-1)}(a_0\ox\dots\ox a_{j-1}a_j \ox\dots\ox a_{n+1})\\
+(-1)^{(j-1+j)}(a_0\ox\dots\ox a_ja_{j+1}\ox\dots\ox a_{n+1})\\
+\sum_{j<k}(-1)^{(j-1+k)}(a_0\ox\dots\ox a_ka_{k+1}\ox\dots\ox a_{n+1}).
\ea
\label{form2}
\eeqlbl
Comparing the exponents of $-1$, we see that $Gd_{B\ox_A B}=-d_BG$
so that $d(G)=0$ on $B_{>0}\ox_A B_{>0}$.
\par
Now consider an element $(a_0)\ox (a_1)\ox (a_2\ox\dots \ox a_{n+1})\in
B_0\ox_A B_{n-1}$, with $n-1>0$.  
Applying $G d_{B\ox_A B}$ to this
element yields (\ref{form1}) where $j=1$, minus the first two summands, and
applying $d_B G$ yields (\ref{form2}) where $j=1$, minus the first summand.  So
applying $d(G)$ to this element yields
\beq
a_0a_1\ox\dots\ox a_{n+1}\in B .
\eeq
Similarly, for the elements $(a_0\ox\dots \ox a_{n-1})\ox (a_{n})\ox
(a_{n+1})$, applying $d(G)$ yields
\beq
(-1)^{n-1+n}a_0\ox\dots\ox a_{n}a_{n+1}=-a_0\ox\dots\ox a_{n}a_{n+1} .
\eeq
In degree $0$ we have
\beq
d(G)( (a_0)\ox (a_1)\ox (a_2) ) =  a_0a_1\ox a_2-a_0\ox a_1a_2 .
\eeq
Comparing these values in the different
cases to our definition of $F_B$ above, we see that $d(G)=F_B$.
\end{proof}

\subsection{Alternate definition of bracket on the Hochschild complex}\label{altC1}

We call a map $\phi:B\ox_A B\to B$ for which $d(\phi):=d_B \phi+ \phi d_{B\ox_A B}
=F_B$ a contracting homotopy for $F_B$.

\begin{definition}[$\phi$-circle operation, $\phi$-bracket]\label{phi-bracket}
Let $\phi\in \Hom_{A^e}(B\ox_A B,B)$ be any contracting homotopy for $F_B$.  The $\phi$-circle operation
$f\circ_\phi g$ on Hochschild cochains is defined as the composite
\beq
f\circ_\phi g:=f\phi (id_B\ox_A g\ox_A id_B)\Delta^{(2)}.
\eeq
The $\phi$-bracket is then defined as the graded commutator
\beq
[f,g]_\phi:=f\circ_\phi g-(-1)^{(|f|-1)(|g|-1)}g\circ_\phi f.
\eeq
\end{definition}
Note that the $G$-circle operation $\circ_G$ is the standard circle
operation and the $G$-bracket $[\ , \ ]_G$ is the standard Gerstenhaber bracket.

\begin{lemma}\label{phi-G} 
Let $\phi$ be  any contracting homotopy for $F_B$. Then the difference
$(\phi-G):B\ox_AB\to B$ is a boundary in the Hom complex.
\end{lemma}
\begin{proof}
The difference is a  cycle, since $d(\phi)=d(G)=F_B$.
Recall that the following map is a quasi-isomorphism, where proj$_{\ast}$ 
is induced by the 
canonical projection of $B$ onto $A$ (considered as a complex in degree 0, with
0 in all other degrees): 
\beq
\Hom_{A^e}(B\ox_A B, B)\stackrel{\mathrm{proj}_\ast}
{\relbar\joinrel\relbar\joinrel\longrightarrow}
\Hom_{A^e}(B\ox_A B, A) .
\eeq
Note that $B\ox_A B$ is also a bimodule resolution of $A$, by the K\"{u}nneth formula, and so the homology of the right hand side is $\Ext_{A^e}(A,A)$.  Since $\Ext_{A^e}(A, A)$ is
$0$ in negative degrees, we
have $\mathrm{H}_{-1}(\Hom_{A^e}(B\ox_A B, B))=0$. So any  cycle in degree $-1$ 
is a boundary.  Consequently, the difference $\phi-G$ is a boundary.
\end{proof}

\begin{proposition}\label{phibrak}
Let $f$ and $g$ be cocycles in the Hochschild cochain complex
$C(A)=\Hom_{A^e}(B,A)$.  Let $\phi$ be a contracting homotopy
for $F_B$.  Then the difference
\beq
f\circ_\phi g-f\circ g
\eeq
is a boundary,  as is the difference
\beq
[f,g]_\phi-[f,g] . 
\eeq
\end{proposition}
\begin{proof}
Take $\tilde{g}$ to be the function $id_B\ox_A g\ox_A id_B$ in $\Hom_{A^e}(B\ot_A B\ot_A B , B\ot_A B)$.
By this notation $\tilde{g}$, in degree $n$, we mean the sum of all  maps on 
$(B\ox_A B\ox_A B)_n$ of the form
$id_i\ot_A g\ot_A id_j$, where $id_i$ is the identity map on $B_i$ and $i+j=n-|g|$. 
Note that the map $\tilde{g}$ is still a cocycle
since $g$ is a cocycle.  Then
\beq
f\circ_\phi g=f\phi \tilde{g}\Delta^{(2)}  , 
\mbox{ while }
f\circ g=f G\tilde{g}\Delta^{(2)}.
\eeq
The difference is given by
\beq
\ba{rl}
f\circ_\phi g-f\circ g&=f\phi \tilde{g}\Delta^{(2)}-fG\tilde{g}\Delta^{(2)}\\
&=f(\phi-G) \tilde{g}\Delta^{(2)}.
\ea
\eeq
By Lemma~\ref{phi-G}, there exists some  map $\psi$
with $d(\psi)=\phi-G$.  Then, since $f$ and $\tilde{g}$ are cocycles,
\beq
\ba{rl}
(-1)^{|f|}d(f\psi \tilde{g}\Delta^{(2)})&=fd(\psi) \tilde{g}\Delta^{(2)}\\
&=f(\phi-G) \tilde{g}\Delta^{(2)}\\
&=f\circ_\phi g-f\circ g , 
\ea
\eeq
whence $f\circ_\phi g-f\circ g$ is seen to be a boundary, as
claimed.  The second statement follows from the first.
\end{proof}

\begin{corollary}\label{corC1}
\begin{enumerate}
\item For any two cocycles $f$ and $g$, the $\phi$-bracket
  $[f,g]_\phi$ is yet another cocycle.  
\item If $f$ or $g$ is a boundary, then so is $[f,g]_\phi$.
\item On cocycles, the $\phi$-bracket is graded anti-commutative up to a
  boundary and also satisfies the Jacobi identity up to a boundary.
\end{enumerate}
\end{corollary}
\begin{proof}
All of these statements follow from the previous proposition and the
fact that these conditions are satisfied by the Gerstenhaber bracket.
\end{proof}

\begin{corollary}\label{corC2}
For any contracting homotopy $\phi$ for $F_B$, the $\phi$-bracket
$[ \ , \ ]_\phi$ induces a graded Lie bracket on the shifted cohomology
\beq
[ \ , \ ]_\phi:\mathrm{HH}(A)[1]\ox \mathrm{HH}(A)[1]\to \mathrm{HH}(A)[1].
\eeq
This bracket agrees with the standard Gerstenhaber bracket on cohomology.
\end{corollary}

\section{Brackets on other cochain complexes}\label{other} 

In this section we define brackets at the cochain level on complexes
other than the Hochschild complex. 
We show that under some conditions, these brackets induce 
precisely the Gerstenhaber bracket on cohomology.
Koszul algebras over $k$ will satisfy these conditions. 
\par
Let $K\to A$ be a projective 
$A$-bimodule resolution of $A$. 
For most of this section we will want $K$ to satisfy some hypotheses which we outline 
next.

\subsection{Hypotheses on the bimodule resolution $K\to A$}
\label{hypotheses}

We assume that the $A$-bimodule resolution $K\to A$ satisfies the following conditions:
\begin{enumerate}
\item[(a)] $K$ admits an embedding $\iota:K\to B$ of complexes of $A$-bimodules for which the following diagram commutes $$\xymatrixrowsep{3mm}\xym{	& B\ar[dr] & \\ K\ar[ur]^\iota \ar[rr]& & A.}$$
\item[(b)] The embedding $\iota$ admits a section $\pi:B\to K$, i.e.\ an $A^e$-chain map $\pi$ with $\pi \iota=id_K$.
\item[(c)] The diagonal map $\Delta_B:B\to B\ox_A B$ preserves $K$, and hence induces a diagonal quasi-isomorphism $\Delta_K:K\to K\ox_A K$.  Equivalently, $K$ comes equipped with a diagonal quasi-isomorphism $\Delta_K:K\to K\ox_A K$ satisfying $\Delta_B \iota=(\iota\ox_A \iota)\Delta_K$.
\end{enumerate}

Practically speaking, the easiest way for condition (b) to be satisfied is for $K$ to be free on some graded base space $W\subset K$ with $W$ mapping to $TA\subset B$ under $\iota$.  Indeed, one can verify that condition (b) holds if and only if the cokernel of each map $\iota_l:K_l\to B_l$ is projective over $A^e$.  So we could, alternately, require that $K$ satisfy the slightly stronger condition

\begin{enumerate}
\item[(b$'$)] $K$ is free on a graded base $W\subset K$ with $\iota(W)\subset TA$ in $B$.
\end{enumerate}

Conditions (a) and (b) can be seen as relatively mild restrictions.  In contrast, condition (c) holds a great deal of significance.  Indeed, it can be shown that if the minimal free bimodule resolution of a connected graded algebra can be made to satisfy (c), then the algebra is Koszul.  This does not mean, however, that non-Koszul algebras have no resolutions satisfying the above conditions, or that the minimal resolution can not be used in some way to compute the Lie bracket.  We will see in Section \ref{cyclic} that we can still use the minimal resolution for (the group algebras of) the cyclic $p$-group in characteristic $p$ to compute the Lie bracket on Hochschild cohomology.
\par
As the above discussion suggests, the Koszul complex of a Koszul algebra does satisfy our conditions (a)--(c).
See, e.g., \cite{BS} or \cite{N} for a discussion of diagonal maps in the case of a Koszul algebra.  
Verification of the other conditions is more straightforward.
We will not need the definition of a Koszul algebra however, as we work in the
general setting of a complex satisfying conditions (a)--(c).
In the next section we give explicitly the example of a polynomial ring, which is
a Koszul algebra.
One can also show that the Koszul resolution of a PBW deformation of a Koszul algebra fits into our framework 
(\cite[Lemma 4.1]{F}, 
\cite[Lemma 6.2]{N}, \cite{Pr}).  
In this case, the diagonal map on $K$ will be induced by the natural comultiplication on the base $W\subset K$ (denoted $\Lambda^\ast$ in \cite{N}).  Localizations of such algebras will also fit into our scheme.  

\begin{remark}
One actually has to replace $B$ with the reduced bar resolution to get (a)--(c) to hold in the case of a non-augmented PBW deformation of a Koszul algebra.  This is, however, a straightforward process.
\end{remark}

\subsection{$\phi$-brackets on $\Hom_{A^e}(K,A)$}\label{best}

For this  subsection, let us fix a resolution $K\to A$ satisfying the hypotheses \ref{hypotheses}(a)--(c).  Let $\mu$ denote the given quasi-isomorphism $\mu:K\to A$.  Then we have the two chain maps $\mu\ox_A id_K:K\ox_A K\to A\ox_A K\cong K$ and $id_K\ox_A \mu:K\ox_A K\to K\ox_A A\cong A$.  We define the chain map $F_K$ as the difference of these two maps, $$F_K:=(\mu\ox_A id_K-id_K\ox_A \mu):K\ox_A K\to K.$$
(The natural isomorphisms $A\ox_A K\cong K$ and $K\ox_A A\cong K$ implicit in the above definition have been omitted from the notation.)   In the case that $K$ satisfies (b$'$), so that $K=A\ox W\ox A$ and the elements in $K$ are given by sums of monomials $a\ox x\ox a'$, we get the elementwise definition of $F_K$ analogous to the one given in (\ref{ogFB}).  In particular, when $K=B$ the two definitions of $F_B$ agree.

\begin{lemma}
The map $F_K:K\ox_AK\to K$ is a boundary in $\Hom_{A^e}(K\ot _A K, K)$.
\label{leamers}
\end{lemma}
\begin{proof} It suffices to check that $F_K$ maps
to $0$ under the quasi-isomorphism
\beq
\Hom_{A^e}(K\ox_A K, K)\stackrel{\mu_\ast}{\relbar
\joinrel\longrightarrow} \Hom_{A^e}(K\ox_A K, A).
\eeq Since $F_K$ is the difference of the two maps $\mu\ox_A id_K$ and $id_K\ox_A \mu$, composed with the isomorphisms $A\ox_A K\cong K$ and $K\ox_A A\cong K$ respectively, the image of $F_K$ is $0$ if and only if the images of these two maps agree.  \par For any complex $M$ of $A$-bimodules, let $\varphi_M:A\ox_A M\overset{\sim}\to M$ and $\varphi'_M:M\ox_A A\overset{\sim}\to M$ denote the standard isomorphisms.  Since these isomorphisms are natural we will have a commutative diagram $$ \xym{K\ox_A K\ar[rr]^(.55){\mu\ox_A id_K} && A\ox_A K\ar[rr]^(.55){id_A\ox_A \mu}\ar[d]^{\varphi_K} && A\ox_A A\ar[d]^{\varphi_A}\\	&& K\ar[rr]^{\mu} && A.}$$ Whence $$\ba{rl}\mu_\ast\big(\varphi_K(\mu\ox_A id_K)\big)&=\mu\varphi_K(\mu\ox_A id_K)\\ &=\varphi_A (id_A\ox_A\mu)(\mu\ox_A id_K)\\ &=\varphi_A(\mu\ox_A \mu).\ea $$ Similarly, we see $\mu_\ast\big(\varphi'_K(id_A\ox_A A)\big)=\varphi'_A(\mu\ox_A \mu)$.  Since there is an equality $\varphi_A=\varphi'_A$, this gives the desired equality $$\mu_\ast\big(\varphi_K(\mu\ox_A id_K)\big)=\mu_\ast\big(\varphi'_K(id_A\ox_A A)\big),$$ and we conclude $\mu_\ast(F_K)=0$.
\end{proof}

Recall that a contracting homotopy for $F_K$ is a  map $\phi:K\ox_A K\to K$ with $d(\phi):=d_K\phi+\phi d_{K\ox_A K}=F_K$.  The lemma allows us to make the following definition.

\begin{definition}[General $\phi$-circle operation, $\phi$-bracket]\label{phibrak2}
Let $\phi$ be a contracting homotopy for $F_K$, and let $\Delta^{(2)}_K$
be a chain map
from $K$ to $K\ot_A K\ot _A K$. 
(Under hypothesis~\ref{hypotheses}(c), we take 
$\Delta^{(2)}_K:=(\Delta_K\ot id_K) \Delta_K = ( id_K \ot \Delta_K)\Delta_K$.)
The
$\phi$-circle product $f\circ_\phi g$ is the composition  
\beq
f\circ_{\phi} g : = 
f\phi(id_K\ox_A g\ox_A id_K)\Delta^{(2)}_K.
\eeq
The $\phi$-bracket is the graded commutator
\beq
[f,g]_\phi :=f\circ_\phi g-(-1)^{(|f|-1)(|g|-1)}g\circ_\phi f.
\eeq
\end{definition}

Suppose that our resolution $K\to A$ satisfies the freeness property (b$'$).  For example, we could take $K$ to be the Koszul resolution of a PBW deformation of a Koszul algebra.  
(See, e.g., \cite{F,Pr}.) 
We can then express the $\phi$-circle operation and bracket on elements in the generating set $x\in W\subset K$ as
\beq
(f\circ_\phi g)(x)=\sum (-1)^{|g||x_1|}f\big(\phi(x_1\ox g(x_2)\ox x_3)\big)
\eeq
and
\begin{eqnarray*}
[f,g]_\phi(x) &= &\sum (-1)^{|g||x_1|}f\big(\phi(x_1\ox g(x_2)\ox x_3)\big)\\
  & & \hspace{.2cm} -(-1)^{(|g|-1)(|f|-1)}\sum(-1)^{|f||x_1|}g\big(\phi(x_1\ox f(x_2)\ox x_3)\big).
\end{eqnarray*}
Here the sum $\sum x_1\ox x_2\ox x_3$ denotes the element $\Delta^{(2)}(x)$, 
which lies in $W\ox W\ox W\subset K^{\ox_A 3}$ by hypothesis.  In the case that $K=B$ and $\phi=G$, the map $\phi$ simply inserts the apparent missing factor $(-1)^{|x_1|}$ in the above expressions.
\par
We will see in Theorem \ref{Gbracket} that the $\phi$-bracket operation preserves cocycles and coboundaries, and that the induced operation on cohomology is precisely the
Gerstenhaber bracket.
The following lemma will be of significance in a moment.

\begin{lemma}\label{GKpi} 
Let us take $G_K:=\pi G(\iota \ox_A \iota):K\ox_A K\to K$, where $G$ is the standard contracting
homotopy for $F_B$ given in Notation \ref{notations1}(2).  Then
\begin{enumerate}
\item $F_K=\pi F_B(\iota \ox_A \iota)$.
\item $d(G_K)=F_K$.
\end{enumerate}
\end{lemma}
\begin{proof}
Statement (1) follows directly from the definitions of $F_K$ and
$F_B$ given above, the commutative diagram $$\xymatrixrowsep{3mm}\xym{	& B\ar[dr] & \\K\ar[ur]^\iota \ar[rr]& & A}$$of hypothesis \ref{hypotheses}(a), and the fact that $\pi\iota=id_K$.  Statement (2) follows
from (1) since we have 
\beq
d(G_K)=d(\pi G(\iota\ox_A \iota))=\pi d(G)(\iota\ox_A \iota)=\pi F_B(\iota\ox_A \iota)=F_K.
\eeq
\end{proof}

Note that, since $\pi:B\to K$ and $\iota:K\to B$ are quasi-isomorphisms, they induce
quasi-isomorphisms on the Hom complexes
\beq
\pi^\ast:\Hom_{A^e}(K,A)\to \Hom_{A^e}(B,A)\ \text{ and }\ \iota^\ast:\Hom_{A^e}(B,A)\to \Hom_{A^e}(K,A).
\eeq
The latter map is simply restriction to $K$.

\begin{proposition}\label{fg} 
Assume hypotheses~\ref{hypotheses}(a)--(c). 
Given $f$ and $g$ in $\Hom_{A^e}(K,A)$ we have an equality of functions
\beq
f\circ_{G_K} g=\iota^\ast(\pi^\ast f\circ \pi^\ast g)
\eeq
and subsequent equality
\beq
[f,g]_{G_K}=\iota^\ast[\pi^\ast f, \pi^\ast g].
\eeq
\label{mainprep}
\end{proposition}

\begin{proof}
Let us simply expand the functions.
\beq
\ba{rlr}
(f\pi\circ g\pi)\iota&=f\pi G (id_B\ox_A g\pi\ox_A id_B)\Delta_B^{(2)}\iota\\
&=f\pi G (id_B\ox_A g\pi\ox_A id_B) \iota^{\ox 3}\Delta^{(2)}_K\\
&=f\pi G (\iota\ox_A g(\pi\iota)\ox_A \iota)\Delta^{(2)}_K\\
&=f(\pi G(\iota\ox_A \iota)) (id_K\ox_A g\ox_A id_K)\Delta^{(2)}_K & (\text{since }\pi\iota=id_K)\\
&=f G_K(id_K\ox_A g\ox_A id_K)\Delta^{(2)}_K\\
&=f\circ_{G_K} g.
\ea
\eeq
The equality of brackets follows from the fact that the bracket is
defined as the graded $\circ$-commutator.
\end{proof}


Let $\phi$ be any contracting homotopy for $F_K$.  By the same proof
as the one given for Proposition \ref{phibrak}, the differences
\beq
f\circ_\phi g-f\circ_{G_K}g\ \text{and}\ [f,g]_\phi-[f,g]_{G_K}
\eeq
will be  boundaries whenever $f$ and $g$ are cocycles in $\Hom_{A^e}(K,A)$.

\begin{theorem}\label{Gbracket}
Suppose $K$ is a bimodule resolution of $A$ satisfying 
hypotheses~\ref{hypotheses}(a)--(c),  
and let $\phi$ be any contracting homotopy for $F_K$.  Let $f$ and $g$ be
cocycles in $\Hom_{A^e}(K,A)$.
\begin{enumerate}
\item The bracket $[f,g]_\phi$ is a cocycle.
\item If $f$ or $g$ is a boundary, then $[f,g]_\phi$ is a boundary.
\item The induced bracket $[ \ , \ ]_\phi:\mathrm{HH}(A)[1]\ox
  \mathrm{HH}(A)[1]\to\mathrm{HH}(A)[1]$ on cohomology agrees with the
  Gerstenhaber bracket.
\end{enumerate}
\label{thm}
\end{theorem}
\begin{proof}
By the discussion preceding the statement of the theorem, we may assume
without loss of generality that $\phi=G_K=\pi G(\iota\ox_A \iota)$.
Now (1) and (2) follow directly from Proposition \ref{mainprep} and
the fact that $\pi^\ast$ and $\iota^\ast$ are quasi-isomorphisms.
Since $id_{\mathrm{H}(\Hom(K,A))}=(\pi\iota)^\ast=\iota^\ast \pi^\ast$,
we see that the induced isomorphisms on homology are in fact mutually inverse.  So
we have
\beq
(\iota^\ast)^{-1}=\pi^\ast:\mathrm{H}(\Hom_{A^e}(K,A))
  \stackrel{\cong}{\longrightarrow} \mathrm{H}(\Hom_{A^e}(B,A)).
\eeq
This isomorphism is one of graded Lie algebras since, according to
Proposition~\ref{fg}, we will have an equality
\beq
\pi^\ast([f,g]_{G_K})=\pi^\ast(\iota^\ast[\pi^\ast f, \pi^\ast g])=(\pi^\ast
\iota^\ast)[\pi^\ast f,\pi^\ast g]=[\pi^\ast f,\pi^\ast g]
\eeq
on cohomology.  Finally,  the homologies
$\mathrm{H}(\Hom_{A^e}(K,A))$ and $\mathrm{H}(\Hom_{A^e}(B,A))$ are
precisely the Hochschild cohomology $\mathrm{HH}(A)$.
\end{proof}

\subsection{Formula for $\phi$}
In general, it may be difficult to find a  map $\phi$ satisfying 
$d(\phi):=d_K\phi+\phi d_{K\ox_A K}= F_K$.  Let us give one method for
constructing such a homotopy that is related to constructions of chain
maps via contracting homotopies (for example, as described in Mac Lane
\cite{MacLane}).
Consider the extended complex $K\to A\to 0$, by which we mean the complex $\cdots \to K_1\to K_0\to A\to 0$.  This complex is acyclic and can, alternatively, be described as the mapping cone of the quasi-isomorphism $K\to A$.

The following lemma is general, that is, it does not require
hypotheses~\ref{hypotheses}, only that $K$ be a free $A$-bimodule 
resolution of $A$, as well as the further hypotheses stated in the lemma.

\begin{lemma}\label{h-phi} 
Suppose $K$ is free on a graded subspace $W\subset K$.  Let $h$ be any $k$-linear contracting homotopy for the identity map on the extended complex $K\to A\to 0$.  Take $\phi_{-1}=0$.  Define $\phi$, in each
degree $i\geq 0$,
as the $A^e$-linear map $\phi_i:(K\ox_A K)_i\to K_{i+1}$ given inductively by
the formula 
\beq
\phi_i|_{W\ox A\ox W}:=h_i((F_K)_i-\phi_{i-1}(d_{K\ox_A K})_i)|_{W\ox A\ox W}.
\eeq
Then $d(\phi)=F_K$.
\label{constructor}
\end{lemma}

\begin{proof}
To simplify notation, take $F=F_K$ and $d=d_K$, or $d_{K\ox_A K}$ when appropriate.  Let us consider $F$ and $\phi$ as maps to the extended complex.  Note that, since $F_0$ has image in the
space of degree $0$ cycles $Z_0(K\to A\to 0)$, this new
version of $F$ will still be a chain map.  Take $\phi_j=0$ for all
negative $j$.  Then for all negative $j$ the equality
$d_{j+1}\phi_j+\phi_{j-1}d_j=F_j$ holds, since both
sides are just $0$.  Now suppose, for a given $i$, that for all $j<i$
the formula $d_{j+1}\phi_j+\phi_{j-1}d_j=F_j$ holds.
Then after restricting to the generating subspace $W\ox A\ox
W\subset K\ox_A K$, 
\beq
\ba{l}
(d_{i+1}\phi_i+\phi_{i-1}d_i)= d_{i+1}h_i(F_i-\phi_{i-1}d_i)+\phi_{i-1}d_i\\
\hspace{5mm}= (F_i-\phi_{i-1}d_i)-h_{i-1}d_i(F_i-\phi_{i-1}d_i)+\phi_{i-1}d_i\\
\hspace{5mm}=F_i-\phi_{i-1}d_i-h_{i-1}d_iF_i+h_{i-1}d_i\phi_{i-1}d_i+\phi_{i-1}d_i\\
\hspace{5mm}=F_i-h_{i-1}d_iF_i+h_{i-1}d_i\phi_{i-1}d_i\\
\hspace{5mm}=F_i-h_{i-1}d_iF_i+h_{i-1}F_{i-1}d_i-h_{i-1}\phi_{i-2}d_{i-1}d_i\\
\hspace{5mm}=F_i-h_{i-1}d_iF_i+h_{i-1}F_{i-1}d_i\\
\hspace{5mm}=F_i ,
\ea
\eeq
whence $d(\phi)=F$.
\end{proof}

We will use the formula of Lemma~\ref{h-phi} in Sections~\ref{snbracket}
and \ref{cyclic}.

\subsection{$\phi$-brackets under weaker conditions}\label{weaker} 
In the remainder of this section, we describe some weaker conditions under which
the conclusion of Theorem~\ref{Gbracket} still holds.
We will need this more general statement in 
Section~\ref{cyclic} below.
For the following lemma, we assume only hypotheses~\ref{hypotheses}(a) and (b),
and we let $\Delta^{(2)}_K$ be a chain map from $K$ to $K\ot _A K \ot_A K$.
(We do not assume that there is a coassociative chain map $\Delta:K\rightarrow
K\ot_A K$ from which $\Delta^{(2)}_K$ is defined.)

Recall that we have 
the canonical contracting homotopy $G_K:=\pi G(\iota\ox_A \iota):
K\ox_A K\to K$ for $F_K$
(as Lemma~\ref{GKpi} does not require hypothesis~\ref{hypotheses}(c)). 

\begin{lemma}\label{lemma-weaker}
Assume hypotheses~\ref{hypotheses}(a) and (b). 
For cocycles $f,g\in \Hom_{A^e}(K,A)$, the difference
\beq
\iota^\ast(\pi^\ast f\circ \pi^\ast g)- f G_K 
(id_K\ox_A g\ox_A id_K) (\pi\ox_A \pi\ox_A \pi)\Delta_B^{(2)}\iota
\eeq
is a boundary.  If $\Delta_K^{(2)} = 
(\pi\ox_A \pi\ox_A \pi)\Delta_B^{(2)}\iota$, 
then
\beq
\iota^\ast(\pi^\ast f\circ \pi^\ast g)-f\circ_\phi g\hspace{5mm}\text{and}\hspace{5mm} i^\ast[\pi^\ast f,\pi^\ast g]-[f,g]_\phi
\eeq
are boundaries.
\end{lemma}
\begin{proof}
We have
\beq
\iota^\ast(\pi^\ast f\circ \pi^\ast g)=f\pi G (id\ox g\pi\ox id)\Delta^{(2)}_B\iota
\eeq
and
\beq
\ba{rl}
f G_K (id_K\ox g\ox id_K) (\pi\ox \pi\ox \pi)\Delta_B^{(2)}\iota&=
  f\pi G(\iota\ox \iota)(\pi\ox g\pi\ox \pi)\Delta^{(2)}_B\iota\\
&= f\pi G(\iota\pi\ox \iota\pi)(id\ox g\pi\ox id)\Delta^{(2)}_B\iota.
\ea
\eeq
Now one can check that $\pi F_B(\iota\pi\ox \iota\pi)=\pi F_B$ 
(since  $\pi\iota=id_K$ and by the definition of $F_B$).  So
\beq
d(\pi G(\iota\pi\ox \iota\pi))=\pi d(G) (\iota\pi\ox \iota\pi)=
\pi F_B(\iota \pi\ox \iota\pi)=\pi F_B=d(\pi G),
\eeq
and, since cohomology vanishes in negative degrees, $\pi G(\iota\pi\ox \iota\pi)-\pi G$ 
is a boundary.  It follows that this difference
\beq
\ba{l}
\iota^\ast(\pi^\ast f\circ \pi^\ast g)- f G_K (id_K\ox_A g\ox_A id_K) 
(\pi\ox_A \pi\ox_A \pi)\Delta_B^{(2)}\iota\\
=f\pi G (id\ox g\pi\ox id)\Delta^{(2)}_B\iota -f\pi G(\iota \pi\ox 
 \iota\pi)(id\ox g\pi\ox id)\Delta^{(2)}_B\iota
\ea
\eeq
is a boundary as well, since all of $f, \pi , \iota , g, \Delta^{(2)}_B$ are cycles. 
\end{proof}


It follows that under hypotheses~\ref{hypotheses}(a) and (b), taking 
 $\Delta_K^{(2)}: = 
(\pi\ox_A \pi\ox_A \pi)\Delta_B^{(2)}\iota$
in Definition~\ref{phibrak2}, 
the conclusion of Theorem~\ref{Gbracket} holds. 
Thus $\phi$-brackets may be defined in a fairly general setting,
at the expense of dealing more directly with maps $\pi$, $\iota$ comparing
to the bar resolution. Note that Definition~\ref{phibrak2} can be used
to define other versions of $\phi$-bracket, given other choices of chain map
$\Delta^{(2)}_K$. At the moment, we do not know which of these other $\phi$-brackets are well-defined
on cohomology, nor whether they have useful properties. 
\par

As we will see in the example of Section \ref{cyclic}, one may be able to produce a satisfactory map $\Delta^{(2)}_K$ without any explicit reference to $\iota$ or $\pi$, and produce a subsequent candidate for the Gerstenhaber bracket.  In the example of Section \ref{cyclic} we check that our map $\Delta^{(2)}_K$ is of the form $\Delta_K^{(2)}: = 
(\pi\ox_A \pi\ox_A \pi)\Delta_B^{(2)}\iota$ for some choice of $\iota$ and $\pi$.

\section{Recovering Schouten-Nijenhuis brackets for  polynomial rings}
\label{snbracket}
\subsection{Review of the Koszul resolution} Let $A=k[x_1,\dots, x_n]$
be the polynomial ring in $n$ variables.  We take $V$ to be the
$k$-vector space with basis $\{x_1,\dots, x_n\}$.  As a formality, let $x_0=1$.  

\begin{definition}
Let $S_i$ denote the symmetric group on $i$ symbols. 
For any $v_1,\dots, v_i\in V$,  let $o(v_1,\dots, v_i)$ denote the
$S_i$-orbit sum
\beq
o(v_1,\dots, v_i)=\sum_{\sigma\in S_i}\mathrm{sgn}(\sigma)v_{\sigma(1)}\ox\dots\ox v_{\sigma(i)}
\eeq
in $V^{\ox n}\subset A^{\ox n}$.  We take $o(\varnothing):=1$.
\end{definition}

Let $W$ denote the graded subspace $\oplus_{i\geq 0} o(V,\dots, V)$ in
$TA$.  One can check that $W$ is a subcoalgebra of $TA$ and that
$K=K(A):=A\ox W\ox A$ is a subcomplex of the bar resolution $B=(A\ox
TA\ox A,d)$.  (See also \cite{BS}, \cite{N}, and (\ref{cmltp}) below.)  It is well known that the embedding
$K\to B$ is a quasi-isomorphism, i.e.\ that $A$ is Koszul.
In the following lemma, the notation $\hat{v}_l$ indicates that 
$v_l$ has been removed.

\begin{lemma}
The differential on $K$ is given on monomials by
\beq
\ba{l}
a\ox o(v_1,\dots, v_i)\ox a'\\
\mapsto \sum_l(-1)^{l+1}a v_l\ox o(v_1,\dots,
\hat{v_l},\dots, v_i)\ox a'-(-1)^{l+1}a \ox o(v_1,\dots,
\hat{v_l},\dots, v_i)\ox v_la' . 
\ea
\eeq
\label{lemmm}
\end{lemma}
\begin{proof}
This follows by  direct computation and the fact that
$v_lv_m-v_mv_l$ is $0$ in $A$ for each $v_l, v_m\in V$.
\end{proof}

We choose the ordering $x_1<x_2<\dots<x_n$ on the generators of $A$ and
call an element 
\beq
1\ox o(x_{i_1},\dots, x_{i_s})\ox x_{j_1}\dots x_{j_t}\ox o(x_{k_1},\dots  , x_{k_u})\ox 1
\eeq
in $k\ox W\ox A\ox W\ox k$ an ordered ``monomial'' if  $x_{i_l}<
x_{i_{l+1}}$, $x_{j_l}\leq x_{j_{l+1}}$, and $x_{k_l}<x_{k_{l+1}}$,
for all $l$.  We define ordered monomials in $A\ox W\ox k$ and in  $k\ox W\ox
A$ similarly.  The $A^e$ generating subspaces $k\ox W\ox k$ and $k\ox
W\ox A\ox W\ox k$, of $K$ and $K\ox_A K$ respectively, are spanned
over $k$ by the respective sets of ordered monomials.
\par
We employ a slight variation of a left $k$-linear contracting homotopy for
the identity on the extended complex $K\to A\to 0$ given in \cite{Witherspoon-Zhou}.
In homological degrees $-1$ and $0$, which are $A$ and $A\ox A$ respectively, $h$ is given by the formula
\beq
h:\ba{l} a\mapsto 1\ox a\\ x_{j_1}\dots x_{j_t}\ox a\mapsto \sum_{1\leq\nu\leq t} x_{j_1}\dots x_{j_{\nu-1}}\ox o(x_{j_\nu})\ox x_{j_{\nu+1}}\dots x_{j_t}a,\ea
\eeq
for any ordered monomial $ x_{j_1}\dots x_{j_t}\ox 1$ in $A\ox A=K_0$
and $a$ in $A$.  In higher
degrees we define $h$ to be the right $A$-linear map specified on
ordered monomials by the formulae
\beq
\ba{c}
h: x_{j_1}\dots x_{j_t}\ox o(x_{k_1},\dots , x_{k_u})\ox 1\\
\mapsto (-1)^u\sum_{x_{j_\nu}>x_{k_u}}x_{j_1},\dots ,
x_{j_{\nu-1}}\ox o(x_{k_1},\dots x_{k_u}, x_{j_\nu})\ox
x_{j_{\nu+1}}\dots x_{j_t}. 
\ea
\eeq
When the indexing set $\{x_{j_\nu}:x_{j_\nu}>x_{k_u}\}$ is empty, the
sum is indeed taken to be $0$.
\par
Using Lemma \ref{constructor} and the contracting homotopy $h$ given
above, one can easily construct a contracting homotopy $\phi:K\ox_A
K\to K$ for $F_K$ in low degrees.  One then deduces from this
information the following general formula.

\begin{definition}
We define the $A^e$-linear map $\phi:K\ox_A K\to K$ on ordered
monomials by the formulas
\\
$\phi: 1\ox x_{j_1}\dots x_{j_t}\ox o(x_{k_1},\dots, x_{k_u})\ox 1$
\beq
\mapsto (-1)^u\sum_{x_{k_u}< x_{j_\nu}} x_{j_1}\dots x_{j_{\nu-1}}\ox o(x_{k_1},\dots , x_{k_u},x_{j_\nu})\ox x_{j_{\nu+1}}\dots x_{j_t}
\eeq
$
\phi:1\ox o(x_{i_1},\dots, x_{i_s})\ox x_{j_1}\dots x_{j_t}\ox 1$
\beq
\mapsto \sum_{x_{j_\nu}<x_{i_1}}x_{j_1}\dots x_{j_{\nu-1}}\ox o(x_{j_\nu}, x_{i_1},\dots, x_{i_s})\ox x_{j_{\nu+1}}\dots x_{j_t}
\eeq
on $K_0\ot_AK_u$ and on $K_s\ot_A K_0$ ($u,s\geq 0$), and\\
$\phi:1\ox o(x_{i_1},\dots, x_{i_s})\ox x_{j_1}\dots x_{j_t}\ox o(x_{k_1},\dots , x_{k_u})\ox 1$
\beqlbl
\mapsto \sum_{x_{k_u}<x_{j_\nu}<x_{i_1}}(-1)^{su+u}x_{j_1}\dots x_{j_{\nu-1}}\ox o(x_{k_1},\dots x_{k_u},x_{j_\nu}, x_{i_1},\dots, x_{i_s})\ox x_{j_{\nu+1}}\dots x_{j_t}.
\label{bla}
\eeqlbl
\end{definition}

\begin{proposition}
The map $\phi:K\ox_A K\to K$ satisfies $d(\phi):=d_K\phi+\phi d_{K\ox_A K}=F_K$.
\label{proporp}
\end{proposition}

We omit the proof of this proposition, which is a delicate, but
straightforward calculation.  

\subsection{Computing the bracket directly from $\phi$ and Theorem
  \ref{Gbracket}}

Before we begin let us make a remark.  In general, one wants to be
strategic in computing the Lie bracket.  One should probably use some
additional structures on Hochschild cohomology, such as the cup
product in combination with (\ref{derivation}), 
additional gradings, etc.  However, we are able to recover here, essentially with a single calculation, a formula for the brackets of cocycles of arbitrary degree via Theorem \ref{Gbracket}.
\par

We will employ the standard isomorphism
\beq
\ba{c}
A[\partial_1,\dots,\partial_n]\to\Hom_{A^e}(K,A)=\mathrm{HH}(A)\\
\partial_i\mapsto (a\ox o(x_j)\ox a'\mapsto \delta_{ij}aa').
\ea
\eeq
Here the generators $\partial_i$ are given degree 1,
and $A[\partial_1,\dots,\partial_n]$ denotes the free graded
commutative $A$-algebra on these generators. 
We identify the monomial $\partial_{i_1}\dots\partial_{i_s}$ 
with the  function dual to the orbit sum $(-1)^{\sum_{\ell=1}^{s-1}\ell}o(x_{i_1},\dots,
x_{i_s})$ in $\Hom_{A^e}(K,A)=\Hom_k(W,A)$.

\par
It will be convenient to have a bit more notation for the statement of
the next proposition.

\begin{notation}
For any ordered set $I=\{i_1,\dots, i_s\}$ of integers satisfying 
$1\leq i_k\leq n$ for all $k$, we take
\beq
\partial_I:=\partial_{i_1}\dots\partial_{i_s}\in \Hom_{A^e}(K,A)
\eeq
and 
\beq
o(x_I)=o(x_{i_1},\dots, x_{i_s})\in W.
\eeq
For $i_k\in I$, we take
\beq
I(k):=\{i_1,\dots, i_{k-1}\}\hspace{5mm}\text{and}\hspace{5mm} I'(k):=\{i_{k+1},\dots, i_l\}.
\eeq
For ordered sets $I$ and $J$ we give $I\amalg J$ the natural ordering
with $i<j$ for each $i\in I, j\in J$.
\end{notation}

In these notations we do {\it not} require that the ordering on $I$ is
such that $i_k<i_{k+1}$ as integers.  We will always let $i_k$ denote
the $k$th element of $I$, as determined by $I$'s given order.  If we
take $\underline{n}=\{1,\dots, n\}$, then the standard $A$-basis for
$A[\partial_1,\dots, \partial_n]=\Hom_{A^e}(K,A)$ can now be written
as the set $\{\partial_I: I\text{ an ordered subset of }\underline{n}\}$.
\par
Via indexing by ordered sets, we can give a clear expression of the
comultiplication on $W$, and the corresponding map $\Delta:K\to K\ox_A
K$.  We have
\beqlbl
\Delta(1\ox o(x_I)\ox 1)=\sum_{I_1,I_2} \pm (1\ox o(x_{I_1})\ox 1)\ox(1\ox o(x_{I_2})\ox 1),
\label{cmltp}
\eeqlbl
where the sum is indexed by all ordered disjoint subsets $I_1,I_2\subset I$
with $I_1\cup I_2=I$, and $\pm$ is the sign of $\sigma$, where
$\sigma$ is the unique permutation with $\{i_{\sigma(1)},\dots
i_{\sigma(|I|)}\}=I_1\amalg I_2$ as an ordered set.

\begin{proposition}
The $\circ_\phi$ operation is given by
\beq
(a\partial_I)\circ_\phi(b\partial_J)=\sum_{1\leq k\leq |I|}(-1)^{(|I|-k)(|J|-1)} a\frac{\partial}{\partial x_{i_k}}(b)\partial_{I(k)\amalg J\amalg I'(k)}
\eeq
and the bracket $[ \ , \ ]_\phi$ is given by
\beq
[a\partial_I,b\partial_J]_\phi=
\eeq
\beq
\sum_{1\leq k\leq |I|}(-1)^{(|I|-k)(|J|-1)} a\frac{\partial}{\partial x_{i_k}}(b)\partial_{I(k)\amalg J\amalg I'(k)}-
\sum_{1\leq l\leq |J|}(-1)^{(l-1)(|I|-1)} b\frac{\partial}{\partial x_{j_l}}(a)\partial_{J(l)\amalg I\amalg J'(l)}.
\eeq
\label{prop300}
\end{proposition}

Note that if $I$ and $J$ share some indices, many of the terms 
$a\frac{\partial}{\partial x_{i_k}}(b)\partial_{I(k)\amalg J\amalg
  I'(k)}$ may be $0$.

\begin{proof}
Take $f=(a\partial_I)$ and $g=(b\partial_J)$ with $I$ and $J$ ordered
subsets of $\underline{n}$.  We may assume $b$ is an ordered monomial
$b=x_{j_1}\dots x_{j_t}$.  We first provide a computation with symbols
$(-1)^{\epsilon_i}$ in place of significant signs.  We will then go
back and provide the appropriate signs.
\par
Suppose we have a nonzero monomial $1\ox o(x_{I(k)\amalg J\amalg
  I'(k)})\ox 1$.  This implies, in particular, that $J$ and $I(k)\cup
I'(k)$ share no indices.  Then 
\beq
\Delta^{(2)}(1\ox o(x_{I'(k)\amalg J\amalg I(k)})\ox 1)=
\eeq
\beq
\ba{l}
(1\ox o(x_{I'(k)})\ox 1)\ox (1\ox o(x_J)\ox 1)\ox (1\ox o(x_{I(k)})\ox 1)\\
+\sum_{\mathscr{S}} \pm (1\ox o(x_{I_2})\ox 1)\ox (1\ox o(x_J)\ox 1)\ox (1\ox o(x_{I_1})\ox 1)\\
+\sum_{J'\neq J} \pm (1\ox o(x_{I'_2})\ox 1)\ox (1\ox o(x_{J'})\ox 1)\ox (1\ox o(x_{I'_1})\ox 1),
\ea
\eeq
where $\mathscr{S}$ is the set of all pairs of subsets $I_1, I_2\subset
I(k)\amalg I'(k)=I\setminus\{i_k\}$ with $I_1\cup
I_2=I\setminus\{i_k\}$, {\it minus} the pair $\{I(k), I'(k)\}$.  We do
not specify the indexing set of the final sum, except to say that
$J'\neq J$.  Since $b\partial_J(1\ox o(x_{J'})\ox 1)=0$ whenever $J'\neq
J$, the above expression gives
\beq
(1\ox (b\partial_J)\ox 1)\Delta^{(2)}\big(1\ox o(x_{I'(k)\amalg J\amalg I(k)})\ox 1\big)=
\eeq
\beq
(-1)^{\epsilon_1} 1\ox o(x_{I'(k)})\ox b\ox o(x_{I(k)})\ox 1+\sum_{\mathscr{S}} \pm (1\ox o(x_{I_2})\ox 1)\ox b\ox (1\ox o(x_{I_1})\ox 1) . 
\eeq
\par
Now, one can conclude from the description of $\mathscr{S}$ that the
maximal element of each $I_1$ is greater than the minimal element of
$I_2$.  So 
\beq
\phi\big(1\ox o(x_{I_1})\ox 1)\ox b\ox (1\ox x_{I_2}\ox 1)\big)=0
\eeq
and
\beq
\ba{l}
\phi(1\ox (b\partial_J)\ox 1)\Delta^{(2)}(1\ox o(x_{I'(k)\amalg J\amalg I(k)})\ox 1)\\
\hspace{10mm}=(-1)^{\epsilon_1}\phi (1\ox o(x_{I'(k)})\ox b\ox o( x_{I(k)})\ox 1)\\
\hspace{10mm}=(-1)^{\epsilon_1}\phi (1\ox o(x_{I'(k)})\ox x_{j_1}\dots x_{j_t}\ox o(x_{I(k)})\ox 1)\\
\hspace{10mm}=(-1)^{\epsilon_2}\sum_{x_{i_{k-1}}<x_{j_\nu}<x_{i_{k+1}}}x_{j_1}\dots x_{j_{\nu-1}}\ox o(x_{I(k)\amalg\{j_\nu\}\amalg I'(k)})\ox x_{j_{\nu+1}}\dots x_{j_t}.
\ea
\eeq
Finally, since $\partial_I(1\ox o(x_{I(k)\amalg\{j_\nu\}\amalg
  I'(k)})\ox 1)=0$ whenever $I(k)\amalg\{j_\nu\}\amalg I'(k)\neq I$,
i.e.\  whenever $j_\nu\neq i_k$, we have
\beq
\ba{l}
\big((a\partial_I)\circ_{\phi}(b\partial_J)\big)(1\ox o(x_{I'(k)\amalg J\amalg I(k)})\ox 1 )\\
\hspace{10mm}=a\partial_I\phi(1\ox (b\partial_J)\ox 1)\Delta^{(2)}(1\ox o(x_{I'(k)\amalg J\amalg I(k)})\ox 1)\\
\hspace{10mm}=(-1)^{\epsilon_3} a\big (|\{\nu: x_{j_\nu}=x_{i_k}\}| \frac{b}{x_{i_k}})\\
\hspace{10mm}=(-1)^{\epsilon_3} a\frac{\partial}{\partial x_{i_k}}(b).
\ea
\eeq
One can check that $(a\partial_I)\circ_{\phi}(b\partial_J)$ vanishes
on all monomials $1\ox o(x_L)\ox 1$ with $L$ not of the form
$I'(k)\amalg J\amalg I(k)$, after some permutation.  So we have
\beq
(a\partial_I)\circ_{\phi}(b\partial_J)=\sum_k (-1)^{\epsilon_4} a\frac{\partial}{\partial x_{i_k}}(b)\partial_{I'(k)\amalg J\amalg I(k)},
\eeq
and, after switching the position of $I(k)$ and $I'(k)$,
\beq
(a\partial_I)\circ_{\phi}(b\partial_J)=\sum_k (-1)^{\epsilon_5} a\frac{\partial}{\partial x_{i_k}}(b)\partial_{I(k)\amalg J\amalg I'(k)}.
\eeq
As for the signs, we have
\beq
\ba{cc}
\ba{l}
\epsilon_1=(|I|-k)|J|+\sum_{\ell=1}^{|J|-1}\ell\\
\epsilon_2=\epsilon_1+(|I|-k)(k-1)+(k-1)\\
\epsilon_3=\epsilon_2+\sum_{\ell'=1}^{|I|-1}\ell'\\
\ea
\ba{l}
\epsilon_4=\epsilon_3-\sum_{\ell''=1}^{|I|+|J|-2}\ell''\\
\hspace{4mm}=\epsilon_3-(|I|-1)(|J|-1)-\sum_{\ell'=1}^{|I|-1}\ell'-\sum_{\ell=1}^{|J|-1}\ell\\
\epsilon_5=\epsilon_4+(|I|-1)|J|+(|I|-k)(k-1)\\
\hspace{4mm}=(|I|-k)(|J|-1)\mod 2.
\ea
\ea
\eeq
The formula for the bracket now follows by the definition of $[f,g]_\phi$.
\end{proof}

For elements $a\partial_i$ and $b\partial_j$ the above formula gives
\beqlbl
[a\partial_i,b\partial_j]_\phi= a\frac{\partial}{\partial
  x_i}(b)\partial_j-b\frac{\partial}{\partial x_j}(a)\partial_i.
\label{eq5000}
\eeqlbl
One can also verify the identity
\beqlbl
[a\partial_I, b_1\partial_{J_1} b_2\partial_{J_2}]_\phi=[a\partial_I,
b_1\partial_{J_1}]_\phi b_2\partial_{J_2} +(-1)^{(|I|-1)|J_1|}b_1\partial_{J_1} [a\partial_I,b_2\partial_{J_2}]_\phi.
\label{eq5001}
\eeqlbl
So the bracket given in Proposition \ref{prop300} is seen to recover
the Schouten-Nijenhuis bracket, which is generally expressed using the
formula (\ref{eq5000}) for the bracket on degree $1$ cocycles and
extended to all of $A[\partial_1,\dots,\partial_i]=\mathrm{HH}(A)$
using the graded derivation identity (\ref{eq5001}).

\section{Recovering  Gerstenhaber brackets for groups of prime order}
\label{cyclic} 

Assume in this section that the characteristic of the field $k$ is $p>0$. 
Let $G$ be a cyclic group of order $p$, with generator $g$,
and  $A=kG$, the group algebra. 
Let $x:= g-1$ in  $A$, so that $A\cong k[x]/(x^p)$.
The Hochschild cohomology of $A$ is well-known.
See \cite{Cibils-Solotar} for the algebra structure of $\HH(kG)$ in
the more general case that $G$ is abelian.
In particular, in that case, 
$\HH(kG)\cong \coh(G,k)\ot kG$ as algebras, where $\coh(G,k)$
denotes group cohomology. 
See Sanchez-Flores \cite{Sanchez-Flores} for the Gerstenhaber brackets 
when $G$ is cyclic; using our new techniques,  
we will recover her results  in our case (i.e.\ $G$ has order $p$).
While the
minimal resolution that we use here does not satisfy all the hypotheses~\ref{hypotheses}(a)--(c) assumed in
Theorem~\ref{Gbracket}, it does satisfy \ref{hypotheses}(a) and (b) (the weaker
conditions assumed in Subsection~\ref{weaker}). 
We will show that 
our alternative approach
yields the Gerstenhaber bracket for these examples.

We will use the following $A^e$-module resolution of $A$ 
(see, e.g., \cite[Exer.\ 9.1.4]{weibel}):
$$
K: \quad \quad 
   \cdots\stackrel{\cdot v}{\longrightarrow} A^e
  \stackrel{\cdot u}{\longrightarrow} A^e
  \stackrel{\cdot v}{\longrightarrow} A^e
  \stackrel{\cdot u}{\longrightarrow} A^e
  \stackrel{m}{\longrightarrow} A \rightarrow 0,
$$
where $u= x\ot 1 - 1\ot x$, $\ v = x^{p-1}\ot 1 + x^{p-2}\ot x + 
\cdots + 1\ot x^{p-1}$, and 
$m$ denotes multiplication.  
For each $i$, let $\xi_i$ denote the element $1\ot 1$ of $A^e$ in degree $i$. 

The following 
maps $h_n: K_n \rightarrow K_{n+1}$ constitute a contracting homotopy
for $K$, as may be verified by direct calculation. 
\begin{eqnarray*}
 h_{-1}(x^i) & = & \xi_0 x^i , \\
  h_0(x^i\xi_0 x^j) & = & 
    \sum_{l=0}^{i-1} x^l \xi_1 x^{i+j-1-l}, \\
h_1(x^i\xi_1 x^j) & = & \delta_{i, p-1} x^j\xi_2 , \\
h_{2n}(x^i\xi_{2n} x^j) &=& 
      - \sum_{l=0}^{j-1} x^{i+j-1-l} \xi_{2n+1} x^l \quad (n\geq 1),\\
h_{2n+1} (x^i\xi_{2n+1} x^j) & = & \delta_{j, p-1} x^i\xi_{2n+1} \quad (n\geq 1) .
\end{eqnarray*} 

Applying Lemma~\ref{h-phi}, we may obtain maps $\phi_n: (K\ot_A K)_n \rightarrow K_{n+1}$,
from the maps $h_n$, for which $d(\phi) = F_K$. 
We only need these maps in degrees 0 and 1: 
\begin{eqnarray*}
\phi_0(\xi_0\ot x^i \xi_0) & = & 
     \sum_{l=0}^{i-1} x^l \xi_1 x^{i-1-l} , \\
\phi_1(\xi_1\ot x^i\xi_0) & = &  - \delta_{i,p-1} \xi_2 , \\
\phi_1(\xi_0\ot x^i \xi_1) & = & \delta_{i,p-1} \xi_2  . 
\end{eqnarray*}

Next we record a diagonal map $\Delta : K \rightarrow K\ot _A K$.
It may be checked directly that the following map is a chain map: 
\begin{eqnarray*}
   \Delta_0(\xi_0) & = & \xi_0\ot \xi_0 , \\
   \Delta_1(\xi_1) & = & \xi_1\ot \xi_0 + \xi_0\ot \xi_1 , \\
   \Delta_2(\xi_2) & = & \xi_2\ot \xi_0 + \xi_0\ot \xi_2 + 
     \sum_{a+b+c=p-2} x^a \xi_1 \ot x^b \xi_1x^c , \\
    \Delta_3(\xi_3) & = & \xi_3\ot \xi_0 + \xi_2\ot \xi_1 + \xi_1\ot \xi_2
       + \xi_0\ot \xi_3 ,\\
\mbox{and generally for }n\geq 1, &  & \\
\Delta_{2n}(\xi_{2n}) & = & \sum_{i=0}^n \xi_{2i}\ot \xi_{2n-2i} + \sum_{i=0}^{n-1}
   \sum_{a+b+c=p-2} x^a \xi_{2i+1}\ot x^b \xi_{2n-2i-1}x^c,\\
\Delta_{2n+1}(\xi_{2n+1}) & = & \sum_{i=0}^{2n+1} \xi_i\ot \xi_{2n+1-i} .
\end{eqnarray*}
Let 
\begin{equation}\label{Delta-2}
\Delta^{(2)}_K:= ( id\ot \Delta) \Delta
\end{equation}
for the purpose of computing $\phi$-brackets under Definition~\ref{phibrak2}.
(Note that $\Delta$ is coassociative if and only if $p=2$.)

We will compute $\phi$-brackets on cohomology in low degrees.
Applying $\Hom_{A^e}( - ,A)$ to $K$, the differentials all are 0.
In each degree, the cohomology is the free $A$-module $A$. 
Let $x^j\xi_i^* \in \Hom_{A^e} (A^e, A)$ denote the function that takes $\xi_i$
to $x^j$. 
Cup products are known: 
If $p=2$, then $\xi_1^*$ generates the Hochschild cohomology as an $A$-algebra
(recall $\HH (A)\cong \coh (G,k)\ot A$),
while if $p>2$, it is generated by $\xi_1^*$ and $\xi_2^*$. 
By applying the identities (\ref{derivation}), 
we need only compute brackets of pairs of elements of degrees 1 and 2.
The $\phi$-circle product of $x^i\xi_1^*$ and $x^j\xi_1^*$ 
in degree 1 is given by 
\begin{eqnarray*}
 ( x^i \xi_1^*  \circ_{\phi}  x^j\xi_1^*) (\xi_1) 
    &= & x^i \xi_1^*(\phi_0 ( x^j\xi_0\ot \xi_0 + \xi_0\ot x^j \xi_0
      + \xi_0\ot \xi_0x^j))\\
   & = & x^i \xi_1^*( \xi_1 x^{j-1} + x \xi_1 x^{j-2} + \cdots
    + x^{j-1} \xi_1) \\
   & = & j x^{i+j-1} .
\end{eqnarray*}
Therefore, by symmetry, we obtain
$$
   [x^i\xi_1^* , x^j \xi_1^*]_{\phi} = (j-i) x^{i+j-1} \xi_1^* .
$$
The $\phi$-circle product of elements in degrees 1 and 2 is given similarly by 
$$
 ( x^i \xi_1^* \circ_{\phi} x^j \xi_2^* )(\xi_2) 
     = j x^{i+j-1} , 
$$
while in the reverse order we have
\begin{eqnarray*}
 ( x^j \xi_2^* \circ_{\phi} x^i \xi_1^* ) ( \xi_2)   
 \!  & = & \! x^j\xi_2^*\big(\sum_{a+b+c = p-2}  \phi_1( -x^a \xi_1 x^{b+i}\ot \xi_0 x^c)
    + \phi_1( \xi_0\ot x^{a+b+i} \ot \xi_1 x^c)\big) \\
 \!  & = & \! \sum_{a+c = i-1} x^{i+j-1} - \sum_{a+b = p-1-i} x^{i+j-1} \\
    & = & (i+p - i) x^{i+j-1} 
         \ \ = \ \ 0 .
\end{eqnarray*}
So $[x^i\xi_1^*, x^j \xi_2^*]_{\phi} = jx^{i+j-1} \xi_2^*$.
Finally, the $\phi$-circle product of two such elements of degree 2 is  
\begin{eqnarray*}
 (x^i \xi_2^* \circ_{\phi} x^j\xi_2^*)(\xi_3) 
   & = & x^i \xi_2^* ( \phi_1(\xi_0\ot x^j \xi_1) + \phi_1 ( \xi_1\ot x^j \xi_0))\\
   &=& x^i \xi_2^*( \delta_{j,p-1} \xi_2 - \delta_{j,p-1} \xi_2) \\
   & = & 0 . 
\end{eqnarray*}
So $[x^i\xi_2^* , x^j \xi_2^* ]_{\phi} = 0$. 

These computations agree with the Gerstenhaber bracket as computed by 
Sanchez-Flores~\cite{Sanchez-Flores},
as well as with direct computations of the Gerstenhaber bracket 
using standard chain maps $\iota$, $\pi$.
Next we will verify the conditions of Lemma~\ref{lemma-weaker} to explain
why these $\phi$-brackets  agree with Gerstenhaber brackets.

Let $\iota: K\rightarrow B$ and $\pi: B\rightarrow K$ be defined
as follows. (See \cite{BACHG} for a more general setting and
\cite[Section 3]{Burciu-Witherspoon} for the maps as below in this specific case.)
The chain map $\iota$ is given by
$$
  \iota_{2l}(\xi_{2l}) = 1\ot \alpha_l \ \ \ \mbox{ and }
  \ \ \ \iota_{2l+1}(\xi_{2l+1}) = 1\ot x\ot \alpha_l 
$$
where $\alpha_0=1$ and if 
$l\geq 1$, 
$$
  \alpha_l = \sum_{ \{ i_1+i_2+\cdots + i_{l+1} = lp-l 
  \mid i_1, i_2, \ldots , i_l \geq 1 \} }
   x^{i_1} \ot x \ot x^{i_2} \ot x \ot\cdots\ot x\ot x^{i_{l+1}} . 
$$
(Note that in the above sum, $i_{l+1}$ can take on the value 0 while
each of $i_1,\ldots,i_l$ must be at least 1.) 
The chain map $\pi$ is given by 
\begin{eqnarray*}
  \pi_{2l} ( 1\!\ot\! x^{i_1}\!\ot\! x^{i_2}\!\ot\! \cdots\!\ot\! x^{i_{2l}}\!\ot\! 1)\!        &=&\!
    \xi_{2l} x^{i_1+i_2-p} x^{i_3+i_4 -p}\cdots x^{i_{2l-1}+i_{2l} -p} , \\
  \pi_{2l+1}(1\!\ot\! x^{i_1}\!\ot\! x^{i_2}\!\ot\! \cdots\!\ot\! x^{i_{2l+1}}\!\ot\!
    1)\! &=&\!
   \sum_{m=0}^{i_1-1} x^m \xi_{2l+1} x^{i_1-m-1} x^{i_2+i_3 - p} x^{i_4+i_5-p}\cdots
   x^{i_{2l}+ i_{2l+1}-p} .
\end{eqnarray*} 
(In the above expressions, any term involving a negative exponent of $x$
should be interpreted as 0.)
Using these maps, we may check directly that 
the map $\Delta^{(2)}_K $ defined by (\ref{Delta-2}) 
satisfies $\Delta^{(2)}_K = (\pi\ot\pi\ot\pi) \Delta^{(2)}_B \iota$. 
Consequently, Lemma~\ref{lemma-weaker} implies that
$\phi$-brackets, defined as above,  coincide with Gerstenhaber brackets,
as we have observed.

\bibliographystyle{plain}

\def\cprime{$'$}

\end{document}